\documentclass[11pt]{amsart}

\usepackage[margin=1in]{geometry}
\usepackage{amsmath,amssymb,amsthm}

\newtheorem{theorem}{Theorem}
\newtheorem{proposition}[theorem]{Proposition}
\newtheorem{lemma}[theorem]{Lemma}

\theoremstyle{remark}
\newtheorem{remark}[theorem]{Remark}

\DeclareMathOperator{\Sp}{Sp}
\DeclareMathOperator{\SL}{SL}
\DeclareMathOperator{\tr}{tr}

\begin{document}

\title[Degree-\(2\) Siegel theta series of extremal even unimodular lattices]
{On the degree-\(2\) Siegel theta series of extremal even unimodular lattices of ranks \(48\), \(72\), \(96\), and \(120\)}

\author[Scott Duke Kominers]{Scott Duke Kominers}
\address{Harvard Business School; Department of Economics and Center of Mathematical Sciences and Applications, Harvard University; and a16z crypto}
\email{kominers@fas.harvard.edu}
\thanks{I used LLMs to assist with computations in the preparation of this article, particularly GPT-5.4/5 Pro and Claude~4.6/7 Opus (both accessed via Poe with the support of Quora, where I am an advisor). I especially appreciate helpful comments from Manabu Oura and a thorough review from Refine.ink.
The problem, methods, and eventual written form are my own; and of course any errors remain my responsibility. This work was conducted while I was visiting the Technological Innovation, Entrepreneurship, and Strategic Management (TIES) Group at the MIT Sloan School of Management; I greatly appreciate their hospitality.}


\begin{quote}
\begin{center}
\emph{The first of a pair of papers in memory of Professor Michio Ozeki.}
\end{center}
\end{quote}

\begin{abstract}
We study degree-\(2\) Siegel theta series of extremal even unimodular lattices from the genus-\(2\) viewpoint initiated by Ozeki. Using Igusa's structure theorem, we define a depth filtration on genus-\(2\) cusp forms, measured by the total degree in \(\chi_{10}\) and \(\chi_{12}\), and relate it to the vanishing of low Fourier--Jacobi coefficients forced by extremality. In ranks \(48\), \(72\), and \(96\), this interaction closes exactly and yields a direct genus-\(2\) proof that the degree-\(2\) theta series is uniquely determined by extremality (conditional on existence in rank~\(96\)).

In rank \(120\) (again conditional on existence), the same argument leaves a one-dimensional residual line spanned by \(\chi_{10}^6\): with \(\chi_{10}\) in the standard integral normalization, any two such degree-\(2\) theta series differ by an integer multiple of \(\chi_{10}^6\).
\end{abstract}

\maketitle

\section{Introduction}

Let \(L\) be an even unimodular lattice of rank \(n\). By the
Mallows--Odlyzko--Sloane bound \cite{MOS},
\[
\min(L) \le 2\left\lfloor \frac{n}{24} \right\rfloor + 2;
\]
if equality holds, then \(L\) is called \emph{extremal}.

For a positive integer \(g\), the degree-\(g\) Siegel theta series of \(L\) is
\[
\Theta_L^{(g)}(Z)
=
\sum_{(x_1,\dots,x_g) \in L^g}
\exp\left(\pi i\,\tr\bigl(((x_i,x_j))\,Z\bigr)\right),
\qquad Z \in \mathbb H_g.
\]
When \(L\) is even unimodular of rank \(n\), the series \(\Theta_L^{(2)}\) is a Siegel modular form of degree \(2\) and weight \(\frac{n}{2}\) on \(\Sp_4(\mathbb Z)\). The Fourier coefficients of \(\Theta_L^{(g)}\) encode the numbers of \(g\)-tuples of lattice vectors with prescribed Gram matrix, and hence record refined geometric information about the configuration of vectors in \(L\).

A central theme in the theory of extremal lattices is that extremality forces unexpectedly rigid behavior in associated automorphic forms. In genus \(2\), this perspective goes back in an essential way to the work of Ozeki, who emphasized that the structure of Siegel modular forms of degree~\(2\), together with the geometric restrictions imposed by extremality, can be used to study theta series in a direct and conceptual way; see, in particular, \cite{OzekiRelation76,OzekiBasis76,OzekiProperty77,Ozeki40,OzekiSM,OzekiLeech,Ozeki48Num}. The present paper follows that viewpoint.

The basic question considered here is: \textit{to what extent is the degree-\(2\) theta series of an extremal even unimodular lattice determined solely by the extremality condition?} In small ranks, the answer is either immediate or already classical. For ranks \(8\) and \(16\), the associated spaces of Siegel modular forms are one-dimensional, so every even unimodular lattice of those ranks must have the same degree-\(2\) theta series \textit{a priori}. In rank \(24\), the unique extremal lattice is the Leech lattice. In rank \(32\), Erokhin \cite{Erokhin} proved the stronger statement that equality of the degree-\(1\) theta series already forces equality of the Siegel theta series through degree \(3\). Peters \cite{Peters} later proved uniqueness of the degree-\(2\) theta series for extremal even unimodular lattices of ranks \(32\), \(48\), and \(56\), together with the analogous conditional uniqueness statement in ranks \(72\) and \(96\); see also \cite{OuraOzeki32Num,Ozeki48Num} for related computations.

By contrast, degree-\(2\) uniqueness fails in rank \(40\). Ozeki \cite{Ozeki40} constructed extremal even unimodular lattices of rank \(40\) with distinct degree-\(2\) theta series, and Peters \cite{Peters} subsequently showed that distinct extremal even unimodular lattices can have the same degree-\(2\) theta series in that rank. Thus the genus-\(2\) rigidity phenomenon is both subtle and strongly rank-dependent.

There are also stronger higher-genus results in certain cases. Salvati Manni \cite{SM} proved degree-\(3\) uniqueness in ranks \(32\) and \(48\), and obtained one-dimensional residual statements in ranks \(56\) and \(72\), where the difference of two degree-\(3\) theta series must be a scalar multiple of \(\chi_{28}\) and \(\chi_{36}\), respectively; see also Ozeki~\cite{OzekiSM,OuraOzeki32Deg4}. As Ozeki observed \cite{OzekiSM}, for rank~\(72\), the degree-\(2\) uniqueness statement can be derived from Salvati Manni's~\cite{SM} degree-\(3\) results by applying the Siegel \(\Phi\)-operator.

Following Ozeki's genus-\(2\) framework, we study \(\Theta_L^{(2)}\) through the structure of the ring of degree-\(2\) Siegel modular forms and the vanishing of low Fourier--Jacobi coefficients forced by the minimum of an extremal lattice. The main observation is that Igusa's generators naturally induce a depth filtration on genus-\(2\) cusp forms, measured by the total degree in \(\chi_{10}\) and \(\chi_{12}\). In weights \(24\), \(36\), and \(48\)---corresponding to ranks \(48\), \(72\), and \(96\), respectively---the vanishing of the first \(t\) positive-index Fourier--Jacobi coefficients forced by extremality exactly matches the maximal possible depth, and so Ozeki's genus-\(2\) strategy yields uniqueness directly.

In weight \(60\), corresponding to rank \(120\), the same depth argument no longer closes completely---it leaves a one-dimensional residual space spanned by \(\chi_{10}^6\). With the standard integral normalization of \(\chi_{10}\), the residual scalar is in fact integral, so the difference of the degree-\(2\) theta series of any two extremal even unimodular lattices of rank \(120\) must be an integer multiple of \(\chi_{10}^6\).

Our main result is thus the following.

\begin{theorem}\label{thm:main}
Let \(L\) and \(L'\) be extremal even unimodular lattices of rank \(24t\). Then:
\begin{enumerate}
\item[(i)] For \(t \in \{2,3,4\}\),
\[
\Theta_L^{(2)} = \Theta_{L'}^{(2)}.
\]
\item[(ii)] For \(t = 5\),
\[
\Theta_L^{(2)} - \Theta_{L'}^{(2)} = c(L,L')\,\chi_{10}^6
\]
for some integer \(c(L,L') \in \mathbb Z\).
\end{enumerate}
Equivalently, the degree-\(2\) Siegel theta series is unique in ranks \(48\) and \(72\); there is at most one degree-\(2\) Siegel theta series arising from an extremal even unimodular lattice of rank \(96\); and any two degree-\(2\) Siegel theta series arising from extremal even unimodular lattices of rank \(120\) differ by an integer multiple of \(\chi_{10}^6\).
\end{theorem}

In ranks \(48\) and \(72\), existence is known \cite{SPLAG,nebe2014fourth,Nebe}, so part~(i) of the theorem gives actual uniqueness there. In rank \(96\), no extremal even unimodular lattice is currently known to exist, so the statement is conditional uniqueness. In rank \(120\), no extremal even unimodular lattice is presently known to exist either, so part~(ii) should be read as a conditional residual statement: any two such lattices, if they exist, yield degree-\(2\) Siegel theta series that differ by at most an integer multiple of \(\chi_{10}^6\). Throughout, \(\chi_{10}\) is taken in the standard integral normalization fixed in \eqref{eq:chi10-normalization} below; the integrality assertion in part~(ii) refers to this normalization.

For ranks \(48\), \(72\), and \(96\), the argument here recovers the genus-\(2\) uniqueness phenomenon identified by Peters \cite{Peters}, but does so in a way that makes the structural mechanism transparent: the depth filtration and the Fourier--Jacobi vanishing forced by extremality fit together exactly in those weights. For rank \(120\), the same mechanism reduces the comparison of two extremal theta series to a single integer scalar---the scalar multiplying \(\chi_{10}^6\) in their difference. To the best of the author's knowledge, the genus-\(2\) residual statement in rank \(120\) is new.

The proof has two parts. First, we develop the depth filtration on Siegel cusp forms and show that vanishing of the first \(t\) Fourier--Jacobi coefficients forces a cusp form to vanish for \(t \in \{2,3,4\}\), while for \(t = 5\) it forces the form to lie in the residual line \(\mathbb C\chi_{10}^6\). Second, we apply this criterion to differences of degree-\(2\) theta series of extremal even unimodular lattices of rank \(24t\) and conclude part~(i) of Theorem~\ref{thm:main} together with the residual statement of part~(ii); integrality of the residual scalar follows from a direct Fourier-coefficient computation for \(\chi_{10}^6\) at \(\bigl(\begin{smallmatrix}6&3\\3&6\end{smallmatrix}\bigr)\), which we carry out as part of the proof.

Thus we isolate a uniform genus-\(2\) mechanism, rooted in Ozeki's viewpoint, that explains why the argument closes in ranks \(48\), \(72\), and \(96\); and in rank \(120\), it reduces the comparison of two extremal theta series to a single integer invariant of the pair.

The paper is organized as follows. In Section~\ref{sec:depth} we develop the depth filtration using Igusa's structure theorem. Section~\ref{sec:fj-prelim} recalls the Fourier--Jacobi expansions of the relevant generators and the associated Jacobi-form input. In Section~\ref{sec:criterion} we analyze the interaction between depth and Fourier--Jacobi vanishing, obtaining the vanishing criterion that drives the proof. In Section~\ref{sec:differences} we apply that criterion to differences of degree-\(2\) theta series of extremal lattices, prove Theorem~\ref{thm:main}, and carry out the Fourier-coefficient computation needed for integrality in rank \(120\). We close with concluding remarks in Section~\ref{sec:remarks}.

\section{Igusa's ring and the depth filtration}\label{sec:depth}

The organization of the proof follows the genus-\(2\) basis viewpoint initiated by Ozeki \cite{OzekiBasis76}.

Let \(M_k(\Sp_4(\mathbb Z))\) denote the space of scalar-valued Siegel modular forms of degree \(2\) and weight \(k\), and let \(S_k(\Sp_4(\mathbb Z))\) denote the associated subspace of Siegel cusp forms. Likewise, we write \(M_k(\SL_2(\mathbb Z))\) for the space of classical modular forms of weight \(k\) on \(\SL_2(\mathbb Z)\), and let \(S_k(\SL_2(\mathbb Z))\) denote the associated subspace of cusp forms.

Let \(e_4\) and \(e_6\) denote the elliptic Eisenstein series on \(\SL_2(\mathbb Z)\) of weights \(4\) and \(6\). We also write
\[
\Delta(\tau)
:=
q\prod_{n \ge 1}(1-q^n)^{24}
=
\frac{e_4(\tau)^3 - e_6(\tau)^2}{1728}
\in S_{12}(\SL_2(\mathbb Z)),
\qquad q = e^{2\pi i\tau}.
\]

Let \(\Phi:M_k(\Sp_4(\mathbb Z)) \longrightarrow M_k(\SL_2(\mathbb Z))\) be the Siegel operator, defined by
\[
(\Phi F)(\tau)
=
\lim_{\Im(\omega) \to \infty}
F\begin{pmatrix}\tau&0\\0&\omega\end{pmatrix}.
\]
If the Fourier--Jacobi expansion of \(F\) is
\[
F\begin{pmatrix}\tau&z\\ z&\omega\end{pmatrix}
=
\sum_{m \ge 0}\phi_m(F)(\tau,z)\,(q')^m,
\qquad q' = e^{2\pi i\omega},
\]
then we have
\[
\Phi(F) = \phi_0(F).
\]

We write \(E_4^{(2)}\) and \(E_6^{(2)}\) for the genus-\(2\) Eisenstein series of weights \(4\) and \(6\), and write \(\chi_{10}\), \(\chi_{12}\) for Igusa's cusp forms of weights \(10\) and \(12\). Igusa proved that
\[
\bigoplus_{\substack{k \ge 0\\ k\ \mathrm{even}}} M_k(\Sp_4(\mathbb Z))
\cong
\mathbb C[E_4^{(2)},E_6^{(2)},\chi_{10},\chi_{12}]
\]
as a graded polynomial ring; see \cite{Igusa62,Igusa64}. Moreover,
\[
\Phi\bigl(E_4^{(2)}\bigr) = e_4,
\qquad
\Phi\bigl(E_6^{(2)}\bigr) = e_6,
\qquad
\Phi(\chi_{10}) = \Phi(\chi_{12}) = 0.
\]
For a degree-\(2\) Siegel modular form \(F\), we write
\[
F(Z) = \sum_{T \ge 0} a(F;T)\,e^{2\pi i\tr(TZ)},
\]
where \(T\) ranges over positive semidefinite half-integral symmetric \(2 \times 2\) matrices and \(a(F;T)\) is the Fourier coefficient of \(F\) at \(T\). We adopt the standard Igusa--Eichler--Zagier normalization of \(\chi_{10}\) and \(\chi_{12}\); see \cite[Section~6]{EZ}. Under this normalization, both forms have integer Fourier coefficients and, in particular,
\begin{equation}\label{eq:chi10-normalization}
a\!\left(\chi_{10};\,\begin{pmatrix}1&1/2\\1/2&1\end{pmatrix}\right) = 1,
\end{equation}
and the index-\(1\) Fourier--Jacobi coefficient of \(\chi_{10}\) is exactly the Eichler--Zagier Jacobi cusp form \(\varphi_{10,1}\) of weight \(10\) and index \(1\), introduced in Section~\ref{sec:fj-prelim}.

\begin{remark}\label{rem:elliptic-basis}
For each integer \(t \ge 0\),
\[
M_{12t}(\SL_2(\mathbb Z))
=
\bigoplus_{j=0}^{t}\mathbb C\,e_4^{\,3t-3j}\Delta^j.
\]
Indeed,
\[
\bigoplus_{k \ge 0} M_k(\SL_2(\mathbb Z)) = \mathbb C[e_4,e_6],
\]
and if \(4a+6b = 12t\), then \(b\) must be even. Thus, every weight-\(12t\) monomial in \(e_4,e_6\) can be rewritten as a polynomial in \(e_4\) and \(e_6^2 = e_4^3 - 1728\Delta\); hence, as a linear combination of the displayed monomials. These monomials have distinct vanishing orders at the cusp \(i\infty\), namely \(0,1,\dots,t\), since
\[
e_4(\tau) = 1+O(q),
\qquad
\Delta(\tau) = q+O(q^2),
\]
and so
\[
e_4(\tau)^{\,3t-3j}\Delta(\tau)^j
=
q^j+O(q^{j+1})
\qquad (0 \le j \le t);
\]
they are therefore linearly independent.
\end{remark}

We define the \emph{depth} of a monomial
\[
(E_4^{(2)})^{\alpha}(E_6^{(2)})^{\beta}\chi_{10}^a\chi_{12}^b,
\]
to be \(a+b\), i.e., the total degree in \(\chi_{10}\) and \(\chi_{12}\).

\begin{lemma}\label{lem:depth-decomp}
Let \(t \in \{2,3,4,5\}\). Every \(f \in S_{12t}(\Sp_4(\mathbb Z))\) can be written uniquely as
\[
f = \sum_{d=1}^{d_{\max}(t)} f_d,
\qquad
d_{\max}(t) =
\begin{cases}
t& t \in \{2,3,4\},\\
6& t = 5,
\end{cases}
\]
where each \(f_d\) is homogeneous of depth \(d\), namely
\[
f_d =
\sum_{\substack{a,b \ge 0\\ a+b=d\\ 10a+12b \le 12t}}
P_{a,b}(E_4^{(2)},E_6^{(2)})\,\chi_{10}^a\chi_{12}^b,
\]
with \(P_{a,b}(X,Y) \in \mathbb C[X,Y]\) weighted-homogeneous of weight \(12t-10a-12b\) for \(\deg X = 4\) and \(\deg Y = 6\). Moreover, when \(t = 5\), the depth-\(6\) part has the form
\[
f_6 = \alpha\chi_{10}^6
\qquad (\alpha \in \mathbb C).
\]
\end{lemma}

\begin{proof}
Because \(f\) is cuspidal, \(\Phi(f) = 0\). We write
\[
f = A\bigl(E_4^{(2)},E_6^{(2)}\bigr)+B,
\]
where every monomial occurring in \(B\) contains at least one factor of \(\chi_{10}\) or \(\chi_{12}\). Applying \(\Phi\) gives
\[
0 = \Phi(f) = A(e_4,e_6).
\]
Now, since
\[
\bigoplus_{k \ge 0} M_k(\SL_2(\mathbb Z)) = \mathbb C[e_4,e_6]
\]
is a polynomial ring, the map
\[
A(X,Y) \longmapsto A(e_4,e_6)
\]
is injective on weighted-homogeneous polynomials. Hence \(A = 0\), so every monomial of \(f\) has positive depth.

Grouping monomials according to their depth---that is, according to the total degree \(a+b\) in \(\chi_{10}\) and \(\chi_{12}\)---yields a decomposition
\[
f = \sum_{d \ge 1} f_d
\]
into depth-homogeneous parts. The weight bound then shows that only the asserted depths can occur:
\begin{itemize}
\item For \(t \in \{2,3,4\}\), a monomial of depth \(d\) has weight at least \(10d\). Since
\[
10(t+1) > 12t
\qquad \text{for \(t \in \{2,3,4\}\)},
\]
no term of depth \(>t\) can occur.

\item For \(t = 5\), depth \(\ge 7\) is impossible because \(10 \cdot 7 > 60\), so depths \(1,\dots,6\) are the only possibilities. Finally, suppose \(t = 5\) and \(a+b = 6\) with \(10a+12b \le 60\). Then
\[
10a+12b = 10(a+b)+2b = 60+2b \le 60,
\]
so necessarily \(b = 0\) and \(a = 6\). The coefficient polynomial then has weight \(0\), and hence is constant; this proves the claim about \(f_6\).
\end{itemize}
Uniqueness follows from the monomial basis of the polynomial ring.
\end{proof}

\begin{remark}\label{rem:explicit}
Lemma \ref{lem:depth-decomp} specializes as follows:
\begin{itemize}
\item For weight \(24\) \((t = 2)\), the possible depths are \(1\) and \(2\), and the depth-\(2\) part is
\[
\alpha E_4^{(2)}\chi_{10}^2+\beta \chi_{12}^2.
\]
The monomial \(\chi_{10}\chi_{12}\) does not occur because it would require a coefficient of weight \(2\), and there is no nonzero weighted-homogeneous polynomial of weight \(2\) in variables of weights \(4\) and \(6\).

\item For weight \(36\) \((t = 3)\), the possible depths are \(1\), \(2\), and \(3\), and the depth-\(3\) part is
\[
\alpha E_6^{(2)}\chi_{10}^3
+\beta E_4^{(2)}\chi_{10}^2\chi_{12}
+\gamma \chi_{12}^3.
\]
This time, the monomial \(\chi_{10}\chi_{12}^2\) does not occur because it would require a coefficient of weight~\(2\).

\item For weight \(48\) \((t = 4)\), the possible depths are \(1\), \(2\), \(3\), and \(4\), and the depth-\(4\) part is
\[
\alpha (E_4^{(2)})^2\chi_{10}^4
+\beta E_6^{(2)}\chi_{10}^3\chi_{12}
+\gamma E_4^{(2)}\chi_{10}^2\chi_{12}^2
+\delta \chi_{12}^4.
\]
And here, the monomial \(\chi_{10}\chi_{12}^3\) does not occur because it would require a coefficient of weight \(2\).

\item Finally, for weight \(60\) \((t = 5)\), the possible depths are \(1\), \(2\), \(3\), \(4\), \(5\), and \(6\), and the depth-\(6\) part is the only depth-\(6\) monomial, i.e.,
\[
\alpha\chi_{10}^6.
\]
\end{itemize}
\end{remark}

\begin{remark}\label{rem:dimensions}
For orientation, the depth decomposition of Lemma~\ref{lem:depth-decomp} refines the cusp form spaces as follows:
\begin{align*}
\dim S_{24}(\Sp_4(\mathbb Z)) &= 5 = 3+2,\\
\dim S_{36}(\Sp_4(\mathbb Z)) &= 13 = 5+5+3,\\
\dim S_{48}(\Sp_4(\mathbb Z)) &= 26 = 7+8+7+4,\\
\dim S_{60}(\Sp_4(\mathbb Z)) &= 46 = 9+11+11+9+5+1,
\end{align*}
where the summands are the dimensions of the depth-\(1,2,\dots,d_{\max}(t)\) components, respectively. Proposition~\ref{prop:criterion} in the sequel eliminates these components one layer at a time.
\end{remark}

\section{Fourier--Jacobi preliminaries}\label{sec:fj-prelim}

We write
\[
Z =
\begin{pmatrix}
\tau & z\\
z & \omega
\end{pmatrix}
\in \mathbb H_2,
\qquad
q' = e^{2\pi i\omega}.
\]
We recall that every \(F \in M_k(\Sp_4(\mathbb Z))\) has a Fourier--Jacobi expansion
\[
F(Z) = \sum_{m \ge 0}\phi_m(F)(\tau,z)\,(q')^m,
\]
where \(\phi_m(F)\) is a Jacobi form of weight \(k\) and index \(m\). The index-\(0\) coefficient is
\[
\phi_0(F) = \Phi(F).
\]
In particular, if \(F\) is cuspidal, then \(\phi_0(F) = 0\); conversely, for the full modular group, \(\phi_0(F) = 0\) implies that \(F\) is cuspidal.

We recall that
\[
E_4^{(2)}(Z) = e_4(\tau)+O(q'),
\qquad
E_6^{(2)}(Z) = e_6(\tau)+O(q').
\]

For \(\chi_{10}\) and \(\chi_{12}\), it is classical that the first nonzero Fourier--Jacobi coefficient has index \(1\); thus (see \cite{Igusa64,EZ})
\[
\chi_{10}(Z) = \psi_{10}(\tau,z)\,q'+O\bigl((q')^2\bigr),
\qquad
\chi_{12}(Z) = \psi_{12}(\tau,z)\,q'+O\bigl((q')^2\bigr),
\]
with \(\psi_{10} \in J^{\mathrm{cusp}}_{10,1}\) and \(\psi_{12} \in J^{\mathrm{cusp}}_{12,1}\).
(Here, we write \(J^{\mathrm{cusp}}_{k,1}\) for the space of Jacobi cusp forms of weight \(k\) and index \(1\).)

The preceding expansions imply the basic principle underlying our depth argument: if
\[
M = (E_4^{(2)})^{\alpha}(E_6^{(2)})^{\beta}\chi_{10}^a\chi_{12}^b
\]
is a monomial of depth \(d = a+b\), then its Fourier--Jacobi expansion has the form
\[
M(Z) = \Psi(\tau,z)\,(q')^d+O\bigl((q')^{d+1}\bigr)
\]
for some Jacobi form \(\Psi\), and in particular \(\phi_m(M) = 0\) for \(m < d\). Indeed, \(E_4^{(2)}\) and \(E_6^{(2)}\) have nonzero index-\(0\) terms, while
\(\chi_{10}\) and \(\chi_{12}\) both begin in index \(1\).

The basic bookkeeping principle is therefore immediate: every factor of \(\chi_{10}\) or \(\chi_{12}\) contributes at least one unit of Jacobi index, whereas the Eisenstein factors may contribute index \(0\).

\begin{lemma}\label{lem:depth-fj}
Let
\[
M = P\bigl(E_4^{(2)},E_6^{(2)}\bigr)\chi_{10}^a\chi_{12}^b,
\qquad d = a+b,
\]
where \(P(X,Y)\) is weighted-homogeneous. Then, we have
\begin{gather*}
\phi_m(M) = 0 \qquad (m < d),\\
\phi_d(M) = P(e_4,e_6)\,\psi_{10}^a\psi_{12}^b.
\end{gather*}
Consequently, if \(f_d\) is the depth-\(d\) component in Lemma~\ref{lem:depth-decomp}, then
\[
\phi_d(f_d)
=
\sum_{\substack{a,b \ge 0\\ a+b=d\\ 10a+12b \le 12t}}
P_{a,b}(e_4,e_6)\,\psi_{10}^a\psi_{12}^b.
\]
\end{lemma}

\begin{proof}
We write
\begin{gather*}
E_4^{(2)}(Z) = e_4(\tau)+\sum_{m \ge 1}\alpha_m(\tau,z)\,(q')^m,
\qquad
E_6^{(2)}(Z) = e_6(\tau)+\sum_{m \ge 1}\beta_m(\tau,z)\,(q')^m,
\\
\chi_{10}(Z) = \psi_{10}(\tau,z)\,q'+O\bigl((q')^2\bigr),
\qquad
\chi_{12}(Z) = \psi_{12}(\tau,z)\,q'+O\bigl((q')^2\bigr),
\end{gather*}
so that
\[
\chi_{10}^a\chi_{12}^b
=
\psi_{10}^a\psi_{12}^b\,(q')^{a+b}
+
O\bigl((q')^{a+b+1}\bigr).
\]
Since any positive-index term chosen from \(E_4^{(2)}\) or \(E_6^{(2)}\) increases the \(q'\)-order, the coefficient of \((q')^d\) in
\[
P\bigl(E_4^{(2)},E_6^{(2)}\bigr)\chi_{10}^a\chi_{12}^b
\]
comes only from taking the index-\(0\) terms \(e_4,e_6\) from the Eisenstein factors. Hence,
\[
\phi_d(M) = P(e_4,e_6)\,\psi_{10}^a\psi_{12}^b,
\]
and \(\phi_m(M) = 0\) for \(m < d\). Summing over the monomials occurring in \(f_d\) gives the final formula.
\end{proof}

By Eichler--Zagier \cite[Theorem~3.6 and the discussion following it]{EZ}, the spaces of Jacobi cusp forms \(J^{\mathrm{cusp}}_{10,1}\) and \(J^{\mathrm{cusp}}_{12,1}\) are one-dimensional. Let \(\varphi_{10,1}\) and \(\varphi_{12,1}\) be generators normalized by
\[
\varphi_{10,1}(\tau,z)
=
q\bigl(\zeta+\zeta^{-1}-2\bigr)+O(q^2),
\qquad
\varphi_{12,1}(\tau,z)
=
q\bigl(\zeta+\zeta^{-1}+10\bigr)+O(q^2),
\]
where \(q = e^{2\pi i\tau}\) and \(\zeta = e^{2\pi iz}\). With our convention from Section~\ref{sec:depth},
\[
\psi_{10} = \varphi_{10,1},
\qquad
\psi_{12} = c'\,\varphi_{12,1}
\]
for some nonzero constant \(c' \in \mathbb C\); see \cite[Section~6]{EZ}.

\begin{lemma}\label{lem:nonproportional-open-set}
There exists a nonempty connected open set \(U \subset \mathbb H\) such that, for every \(\tau \in U\),
\begin{enumerate}
\item \(\psi_{12}(\tau,\cdot) \not\equiv 0\);
\item the functions \(z \mapsto \psi_{10}(\tau,z)\) and \(z \mapsto \psi_{12}(\tau,z)\) are not proportional.
\end{enumerate}
\end{lemma}

\begin{proof}
By construction, \(\psi_{10} = \varphi_{10,1}\) and \(\psi_{12}\) is a nonzero scalar multiple of \(\varphi_{12,1}\), with first \(q\)-coefficients
\[
q\bigl(\zeta+\zeta^{-1}-2\bigr)
\qquad \text{and} \qquad
q\bigl(\zeta+\zeta^{-1}+10\bigr),
\]
respectively. These leading coefficients are not proportional as functions of \(z\). Indeed, if we evaluate at \(z = \frac12\) and \(z = \frac13\), then
\[
\zeta+\zeta^{-1} = -2 \quad \text{for } z = \frac12,
\qquad
\zeta+\zeta^{-1} = -1 \quad \text{for } z = \frac13.
\]
Hence, as \(\Im(\tau) \to \infty\),
\begin{gather*}
\psi_{10}\left(\tau,\frac12\right) = q\bigl(-4+O(q)\bigr),
\qquad
\psi_{12}\left(\tau,\frac12\right) = c'\,q\bigl(8+O(q)\bigr),\\
\psi_{10}\left(\tau,\frac13\right) = q\bigl(-3+O(q)\bigr),
\qquad
\psi_{12}\left(\tau,\frac13\right) = c'\,q\bigl(9+O(q)\bigr),
\end{gather*}
for some nonzero constant \(c' \in \mathbb C\). Therefore the holomorphic function
\[
R(\tau)
:=
\psi_{10}\left(\tau,\frac12\right)\psi_{12}\left(\tau,\frac13\right)
-
\psi_{10}\left(\tau,\frac13\right)\psi_{12}\left(\tau,\frac12\right)
\]
satisfies
\[
R(\tau) = c'\,q^2\bigl((-4)(9)-(-3)(8)+O(q)\bigr)
=
-12c'\,q^2+O(q^3),
\]
so \(R\) is not identically zero. Hence \(R(\tau) \neq 0\) on a nonempty open subset of \(\mathbb H\), and for such \(\tau\) the functions \(z \mapsto \psi_{10}(\tau,z)\) and \(z \mapsto \psi_{12}(\tau,z)\) are not proportional. Shrinking the open set further if necessary, we may also require that \(\psi_{12}(\tau,\frac12) \neq 0\), so that \(\psi_{12}(\tau,\cdot) \not\equiv 0\); this yields an open set. Taking \(U\) to be a sufficiently small connected open subset on which \(R(\tau) \neq 0\) and \(\psi_{12}(\tau,\frac12) \neq 0\) gives the result.
\end{proof}

\begin{lemma}\label{lem:binary-powers}
Let \(d \ge 0\), let \(V \subset \mathbb C\) be a connected open set, and let
\(A,B\) be meromorphic functions on \(V\). Assume that \(B \not\equiv 0\) and that
\(A\) and \(B\) are not proportional. Then the functions
\[
A^aB^{d-a} \qquad (0 \le a \le d)
\]
are linearly independent over \(\mathbb C\).
\end{lemma}

\begin{proof}
Since \(A\) and \(B\) are not proportional, the meromorphic function
\[
u := A/B
\]
is nonconstant. Hence there exists a point \(z_0 \in V\) away from the poles of \(u\) and zeros of \(B\) at which \(u'(z_0) \neq 0\). On a sufficiently small disk \(W \subset V\) centered at \(z_0\),
the function \(u\) is holomorphic and nonconstant, and \(B\) has no zeros.
Suppose that
\[
\sum_{a=0}^{d} c_a A^aB^{d-a} = 0
\]
identically on \(V\). Dividing by \(B^d\) on \(W\) gives
\[
\sum_{a=0}^{d} c_a\,u^a = 0
\qquad \text{on } W.
\]
Since \(u(W)\) is open, the polynomial \(\sum_{a=0}^{d} c_a X^a\) vanishes on an
open subset of \(\mathbb C\), and therefore all the \(c_a\) must be zero.
\end{proof}

\section{Fourier--Jacobi vanishing and the residual line}\label{sec:criterion}

We now combine the depth decomposition of Lemma~\ref{lem:depth-decomp} with Lemmas~\ref{lem:nonproportional-open-set} and~\ref{lem:binary-powers} to obtain the criterion that drives the proof of the main theorem.

\begin{proposition}\label{prop:criterion}
Let \(t \in \{2,3,4,5\}\), and let \(f \in S_{12t}(\Sp_4(\mathbb Z))\). If
\[
\phi_1(f) = \phi_2(f) = \cdots = \phi_t(f) = 0,
\]
then we have
\begin{align*}
f &= 0 \qquad (t \in \{2,3,4\}),\\
f &\in \mathbb C\chi_{10}^6 \qquad (t = 5).
\end{align*}
\end{proposition}

\begin{proof}
By Lemma~\ref{lem:depth-decomp}, we may write
\[
f = f_1+\cdots+f_{d_{\max}(t)},
\]
where \(f_d\) has depth \(d\). Let \(U \subset \mathbb H\) be the nonempty open set furnished by Lemma~\ref{lem:nonproportional-open-set}.

We show by induction on \(d = 1,\dots,t\) that \(f_d = 0\). Suppose first the inductive hypothesis that \(f_1 = \cdots = f_{d-1} = 0\) (vacuously if \(d = 1\)). Then, by Lemma~\ref{lem:depth-fj}, the \(d\)-th Fourier--Jacobi coefficient of \(f\) is exactly the \(d\)-th Fourier--Jacobi coefficient of \(f_d\), since terms of depth \(<d\) are absent and terms of depth \(>d\) contribute only to indices \(>d\). Thus,
\[
\phi_d(f) = \phi_d(f_d) = 0.
\]

We write
\[
f_d =
\sum_{\substack{a,b \ge 0\\ a+b=d\\ 10a+12b \le 12t}}
P_{a,b}(E_4^{(2)},E_6^{(2)})\chi_{10}^a\chi_{12}^b,
\]
and put
\[
p_{a,b}(\tau) := P_{a,b}(e_4(\tau),e_6(\tau)).
\]
Then Lemma~\ref{lem:depth-fj} gives
\[
0 = \phi_d(f) =
\sum_{\substack{a,b \ge 0\\ a+b=d\\ 10a+12b \le 12t}}
p_{a,b}\,\psi_{10}^a\psi_{12}^b.
\]

We assume for the sake of seeking a contradiction that \(f_d \neq 0\). Then not all \(P_{a,b}\) vanish, and because---as in the proof of Lemma~\ref{lem:depth-decomp}---the map \(P(X,Y) \mapsto P(e_4,e_6)\) is injective on weighted-homogeneous polynomials, not all \(p_{a,b}\) vanish. Since not all \(p_{a,b}\) vanish identically, the set
\[
Z := \{\tau \in U : p_{a,b}(\tau)=0\text{ for all admissible \((a,b)\)}\}
\]
is a proper analytic subset of \(U\). In particular, \(U \setminus Z \neq \varnothing\), so we may choose \(\tau \in U\) such that not all \(p_{a,b}(\tau)\) vanish.

Now, we define
\[
A(z) := \psi_{10}(\tau,z),
\qquad B(z) := \psi_{12}(\tau,z).
\]
By Lemma~\ref{lem:nonproportional-open-set}, \(A\) and \(B\) are not proportional and \(B \not\equiv 0\). Evaluating the preceding identity at our fixed \(\tau\) yields an identity of meromorphic functions of \(z\):
\[
0 =
\sum_{\substack{a,b \ge 0\\ a+b=d\\ 10a+12b \le 12t}}
p_{a,b}(\tau)A(z)^aB(z)^b.
\]
By Lemma~\ref{lem:binary-powers}, the full family \(\{A^a B^{d-a}\}_{0 \le a \le d}\) is linearly independent over \(\mathbb C\); the subfamily indexed by pairs \((a,b)\) with \(a+b = d\) and \(10a+12b \le 12t\) is therefore also linearly independent. Hence all coefficients \(p_{a,b}(\tau)\) in the displayed identity vanish, contradicting our choice of \(\tau\). Hence \(f_d = 0\).

Induction now gives \(f_1 = \cdots = f_t = 0\). If \(t \in \{2,3,4\}\), then Lemma~\ref{lem:depth-decomp} shows that no further depth can occur, so \(f = 0\). If \(t = 5\), only the depth-\(6\) term remains, and Lemma~\ref{lem:depth-decomp} shows that it is a scalar multiple of \(\chi_{10}^6\), as claimed.
\end{proof}

\section{Differences of degree-\(2\) theta series of extremal lattices}\label{sec:differences}

We now verify that extremality forces the vanishing hypothesis in Proposition~\ref{prop:criterion}, and then apply the proposition to differences of degree-\(2\) theta series.

\begin{lemma}\label{lem:fj-vanish}
Let \(t \in \{2,3,4,5\}\), and let \(L\) be an extremal even unimodular lattice of rank \(24t\). Then
\[
\phi_m(\Theta_L^{(2)}) = 0
\qquad (1 \le m \le t).
\]
Consequently, if \(L\) and \(L'\) are extremal even unimodular lattices of rank \(24t\), then
\[
\phi_m\bigl(\Theta_L^{(2)}-\Theta_{L'}^{(2)}\bigr) = 0
\qquad (1 \le m \le t).
\]
\end{lemma}

\begin{proof}
The Fourier coefficients of \(\Theta_L^{(2)}\) are indexed by half-integral symmetric matrices
\[
T =
\begin{pmatrix}
n & r/2\\
r/2 & m
\end{pmatrix}
\ge 0
\]
and count pairs \((x,y) \in L^2\) with Gram matrix \(2T\). In particular, for a Fourier coefficient occurring in the index-\(m\) Fourier--Jacobi coefficient, the second lattice vector \(y\) necessarily satisfies \((y,y) = 2m\); for \(m > 0\), such a vector \(y\) is necessarily nonzero.

Since \(L\) is extremal of rank \(24t\), we have
\[
\min(L) = 2t+2.
\]
If \(1 \le m \le t\), then
\[
2m \le 2t < 2t+2 = \min(L),
\]
so no such vector \(y\) can exist. Hence every Fourier coefficient of Jacobi index \(m\) vanishes for \(1 \le m \le t\), which is exactly the claimed statement that
\[
\phi_m\bigl(\Theta_L^{(2)}\bigr) = 0
\qquad (1 \le m \le t).
\]
The final assertion follows by subtraction.
\end{proof}

With Lemma~\ref{lem:fj-vanish} in hand, we obtain a sharp characterization of \(\Theta_L^{(2)}-\Theta_{L'}^{(2)}\) for any pair \(L,L'\) of extremal even unimodular lattices of rank \(24t\).

\begin{proposition}\label{prop:reduction}
Let \(t \in \{2,3,4,5\}\), and let \(L\) and \(L'\) be extremal even unimodular lattices of rank \(24t\). Then
\[
\Theta_L^{(2)}-\Theta_{L'}^{(2)}
=
\begin{cases}
0& t \in \{2,3,4\},\\[4pt]
c\chi_{10}^6 \text{ for some \(c \in \mathbb C\)}& t = 5.
\end{cases}
\]
\end{proposition}

\begin{proof}
First we consider the ordinary (i.e., degree-\(1\)) theta series. Since \(L\) is extremal, \(\min(L) = 2t+2\), so
\[
\Theta_L^{(1)}(\tau) = 1+O(q^{t+1}),
\qquad q = e^{2\pi i\tau},
\]
and likewise
\[
\Theta_{L'}^{(1)}(\tau) = 1+O(q^{t+1}).
\]
Thus,
\[
g(\tau) := \Theta_L^{(1)}(\tau)-\Theta_{L'}^{(1)}(\tau)
\]
is a modular form of weight \(12t\) on \(\SL_2(\mathbb Z)\) vanishing to order at least \(t+1\) at \(i\infty\).

On the other hand, as noted in Remark~\ref{rem:elliptic-basis},
\[
M_{12t}(\SL_2(\mathbb Z))
=
\bigoplus_{j=0}^{t}
\mathbb C\,e_4^{\,3t-3j}\Delta^j,
\]
and the basis elements have distinct orders \(0,1,\dots,t\) at \(i\infty\). Thus, every nonzero element of \(M_{12t}(\SL_2(\mathbb Z))\) vanishes at \(i\infty\) to order at most \(t\). We must therefore have \(g = 0\); equivalently,
\[
\Theta_L^{(1)} = \Theta_{L'}^{(1)}.
\]

Now, we set
\[
f := \Theta_L^{(2)}-\Theta_{L'}^{(2)}.
\]
Because \(\Phi(\Theta_L^{(2)}) = \Theta_L^{(1)}\) and \(\Phi(\Theta_{L'}^{(2)}) = \Theta_{L'}^{(1)}\), we have \(\Phi(f) = 0\), so \(f \in S_{12t}(\Sp_4(\mathbb Z))\). By Lemma~\ref{lem:fj-vanish},
\[
\phi_1(f) = \cdots = \phi_t(f) = 0.
\]
Proposition~\ref{prop:criterion} therefore applies and gives the claimed characterization.
\end{proof}

We can now complete the proof of Theorem~\ref{thm:main}. Part~(i) is the case \(t \in \{2,3,4\}\) of Proposition~\ref{prop:reduction}. For part~(ii), Proposition~\ref{prop:reduction} gives
\[
\Theta_L^{(2)}-\Theta_{L'}^{(2)} = c\,\chi_{10}^6
\]
for some \(c \in \mathbb C\); it remains to show \(c \in \mathbb Z\). The Siegel theta series \(\Theta_L^{(2)}\) and \(\Theta_{L'}^{(2)}\) both have integer Fourier coefficients (they are generating functions counting lattice pairs), so the same is true of the Fourier coefficients of their difference. Hence \(c \in \mathbb Z\) will follow once we exhibit a half-integral symmetric \(T\) at which \(a(\chi_{10}^6;T) = 1\); the required \(T\) is supplied by the next lemma.

\begin{lemma}\label{lem:chi106-T66}
Let
\[
T = \begin{pmatrix}6 & 3\\ 3 & 6\end{pmatrix}.
\]
Then \(a(\chi_{10}^6;T) = 1\).
\end{lemma}

\begin{proof}
By Lemma~\ref{lem:depth-fj}, the depth-\(6\) component of the depth decomposition of \(\chi_{10}^6\) is \(\chi_{10}^6\) itself, and
\[
\phi_6(\chi_{10}^6) = \psi_{10}^6.
\]
Under the Fourier--Jacobi correspondence,
\[
q^n\,\zeta^r\,(q')^m
\quad \longleftrightarrow \quad
\begin{pmatrix} n & r/2\\ r/2 & m\end{pmatrix},
\]
so \(a(\chi_{10}^6;T)\) is the coefficient of \(q^6\zeta^6\) in \(\psi_{10}^6\). Under our normalization \eqref{eq:chi10-normalization}, \(\psi_{10} = \varphi_{10,1}\), so this is the coefficient of \(q^6\zeta^6\) in \(\varphi_{10,1}^6\).

We compute this coefficient directly from the leading expansion of \(\varphi_{10,1}\). Recall that
\[
\varphi_{10,1}(\tau,z) = q(\zeta+\zeta^{-1}-2)+O(q^2).
\]
Since \(\varphi_{10,1}\) has initial \(q\)-order \(1\), only the leading \(q\)-term of each factor can contribute to the coefficient of \(q^6\) in \(\varphi_{10,1}^6\).
Factoring, we have
\[
\zeta+\zeta^{-1}-2 = \zeta^{-1}(\zeta-1)^2,
\]
so that
\[
\varphi_{10,1}^6 = q^6\zeta^{-6}(\zeta-1)^{12}+O(q^7).
\]
Expanding
\[
(\zeta-1)^{12} = \sum_{k=0}^{12}\binom{12}{k}(-1)^{12-k}\zeta^k,
\]
we read off
\[
[q^6\zeta^r]\,\varphi_{10,1}^6
=
(-1)^{12-(r+6)}\binom{12}{r+6}
=
(-1)^r\binom{12}{r+6}
\qquad (-6 \le r \le 6).
\]
In particular, the coefficient of \(q^6\zeta^6\) in \(\varphi_{10,1}^6\) is \(\binom{12}{12} = 1\). Therefore \(a(\chi_{10}^6;T) = 1\).
\end{proof}

This completes the proof of Theorem~\ref{thm:main}: comparing Fourier coefficients of \(\Theta_L^{(2)}-\Theta_{L'}^{(2)} = c\,\chi_{10}^6\) at \(T = \bigl(\begin{smallmatrix}6&3\\3&6\end{smallmatrix}\bigr)\) gives
\[
c = a\bigl(\Theta_L^{(2)}-\Theta_{L'}^{(2)};T\bigr) \in \mathbb Z.
\]

\section{Concluding remarks}\label{sec:remarks}

Our analysis provides a unified genus-\(2\) depth argument that treats all cases \(t \in \{2,3,4\}\), i.e., ranks \(48\), \(72\), and \(96\). This should be contrasted with rank~\(40\), where degree-\(2\) theta series are known \textit{not} to be unique among extremal lattices \cite{Ozeki40,Peters}; indeed, Peters~\cite{Peters} showed that the degree-\(2\) theta series does not determine the isometry class there. (The rank-\(24\) case is automatic because the Leech lattice is unique.)

For \(t \ge 5\), the depth argument alone no longer closes: the maximum
possible depth allowed by the weight bound is \(\lfloor 6t/5\rfloor\), which
exceeds \(t\) once \(t \ge 5\). In rank \(120\) (i.e., \(t = 5\)), this is the source of the one-dimensional residual line \(\mathbb C\chi_{10}^6\) in Theorem~\ref{thm:main}(ii). Resolving whether the residual scalar \(c(L,L')\) is in fact always zero---equivalently, whether \(\Theta_L^{(2)}\) is unique among extremal lattices of rank \(120\)---is the natural next question. Several avenues for such input seem plausible.

First, the unconstrained residual at depth \(t+1\) is quite small: for \(t = 5\) (rank \(120\), weight \(60\)), only the depth-\(6\) component survives, and it is a scalar multiple of \(\chi_{10}^6\). It may be possible to rule out this component by analyzing the Fourier--Jacobi coefficient \(\phi_{t+1}\) more carefully---not merely checking whether it vanishes, but exploiting the fact that the genus-\(1\) theta series of an extremal lattice is already determined, which constrains the diagonal restriction of \(\phi_{t+1}\) and hence the available Jacobi forms of index \(t+1\).

Second, one may ask how the residual line \(\mathbb C\chi_{10}^6\) sits relative to the Saito--Kurokawa subspace of \(S_{60}(\Sp_4(\mathbb Z))\). Two distinct questions can be raised here. The first is whether \(\chi_{10}^6\) itself satisfies the Maass relations, i.e., whether \(\chi_{10}^6\) lies in the Saito--Kurokawa subspace; the second, and weaker, question is whether \(\chi_{10}^6\) has a nonzero projection there. The relevant approach is to compute, or prove the vanishing of, the actual Saito--Kurokawa projection of \(\chi_{10}^6\), equivalently its Petersson pairings with a basis of Saito--Kurokawa lifts in \(S_{60}(\Sp_4(\mathbb Z))\). If the projection vanishes, then the residual lies entirely in the non-Saito--Kurokawa part, where we may hope for stronger vanishing or slope bounds.

Third, and more ambitiously, one could pass to higher genus. Salvati Manni's genus-\(3\) result for rank \(72\) \cite{SM} already demonstrates that increasing the degree of the Siegel theta series introduces additional constraints. A systematic study of the interplay between the genus-\(2\) and genus-\(3\) depth filtrations, mediated by the \(\Phi\)-operator, could yield uniqueness statements in ranks where neither genus alone suffices. Ozeki's own work \cite{OzekiSM} on the passage from genus~\(3\) to genus~\(2\) provides a natural starting point for such an investigation.

\providecommand{\bysame}{\leavevmode\hbox to3em{\hrulefill}\thinspace}
\providecommand{\MR}{\relax\ifhmode\unskip\space\fi MR }
\providecommand{\MRhref}[2]{%
  \href{http://www.ams.org/mathscinet-getitem?mr=#1}{#2}
}
\providecommand{\href}[2]{#2}

\end{document}